\documentclass[1p]{elsarticle}                           
 \usepackage{amssymb,amsmath,mathrsfs,graphicx,subfigure,multicol,setspace,amsthm}

\usepackage{tikz,times}
\usetikzlibrary{arrows}
 \newcommand{\R}{{\rm \bf R}}
\newcommand{\co}{\mathop{\mbox{co}}}

\newcommand{\Rnn}{\bar{\mathbb R}_+^n}

\newcommand{\diag}{\mathop{\mbox{diag}}}
\newcommand{\sgn}{\mathop{\mbox{sgn}}}
\newcommand{\supp}{\mathop{\mbox{supp}}}
\newcommand{\im}{\mathop{\mbox{Im}}}
\newcommand{\tV}{{\tilde{V}}}
\newcommand{\hV}{{\hat{V}}}
\newcommand{\cW}{{\mathcal W}}

  \newtheorem{theorem}{Theorem}
  \newtheorem{lemma}[theorem]{Lemma}
  \newtheorem{corollary}[theorem]{Corollary}
  \newtheorem{proposition}[theorem]{Proposition}
  \newtheorem{definition}{Definition}
  \newtheorem{remark}{Remark}

\begin{document}
\begin{frontmatter}
\title{
{\bf Robust Lyapunov Functions for Reaction Networks:\\ An Uncertain System Framework}}

 \author[mit]{M. Ali Al-Radhawi}
 \ead{{malirdwi@mit.edu}}
\author[ic,fl]{David Angeli}
 \ead{{david.angeli@imperial.ac.uk}}
  \address[mit]{Department of Mechanical Engineering, Massachusetts Institute of Technology, 77 Massachusetts Avenue, Cambridge, MA 02139, United States.}
 \address[ic]{Dept. of Electrical and Electronic Engineering, Imperial College London, London SW7 2AZ, United Kingdom.}
\address[fl]{Dip. di Ingegneria dell'Informazione, University of Florence, 50139 Florence, Italy.}

\begin{abstract}
We present a framework to transform the problem of finding a Lyapunov function for a Chemical Reaction Network (CRN) with arbitrary monotone kinetics expressed in concentration coordinates into finding a common Lyapunov function for a linear differential inclusion in reaction coordinates. Alternative formulations in different coordinates are also provided.
This is applied to reinterpret previous results by the authors on Piecewise Linear in Rates Lyapunov functions and to establish a link with contraction analysis. The framework is then applied to derive powerful results on the network's persistence and the uniqueness of equilibria.
\end{abstract}

\end{frontmatter}

\section{Introduction}
{Chemical Reaction Networks} (CRNs) are used as models in several areas of science and engineering; including chemical engineering, population dynamics, and {molecular systems biology}. The last application has drawn the recent interest of the systems and control community to this field \cite{sontag01,angeli09tut}.

One of the central dilemmas in systems biology is that the kinetic information required to construct a detailed mathematical model is scarce and subject to uncertainty. This is unlike the wealth of information available on the graphical description of the networks involved. Hence, the main focus of Chemical Reaction Networks Theory and of our research has been to draw conclusions about the asymptotic behaviour of networks classes regardless of detailed knowledge of its kinetics. This has been advocated towards the aim of constructing  ``complex biology without parameters'' \cite{bailey01}. The presumed feasibility of this goal has been motivated by the observation that large classes of CRNs converge asymptotically to steady states regardless of the model of kinetics involved. Indeed, partial success has been achieved in this pursuit. For example, the class of weakly reversible zero-deficiency networks with Mass-Action kinetics has been shown to have unique asymptotically stable equilibria in the interior of the orthant \cite{horn72,feinberg95}. Monomolecular networks, on the other hand, have also been handled within the framework of compartmental systems \cite{maeda78}. The theory of monotone systems has been also applied to provide graphical conditions for global convergence of some of these networks \cite{angeli10}. Recently, piecewise linear Lyapunov functions based methods has been proposed recently \cite{blanchini11,mtns,MA_cdc13,PWLRj,blanchini14}.

In a previous work \cite{MA_cdc13,PWLRj}, the authors proposed a direct approach to the problem, where Piecewise-Linear in Rates (PWLR) Lyapunov functions have been introduced. In addition to their simple structure, these functions are robust with respect to arbitrary variations of kinetic constants, and only require mild assumptions on the reaction rates, with mass-action kinetics being a special case.

In this paper, we generalize this approach by transforming the problem of finding a Lyapunov function expressed in concentration coordinates for networks with arbitrary monotone kinetics into finding a common Lyapunov function for a linear parameter varying system written in reaction coordinates. Furthermore, several related Lyapunov functions are presented.  As a result, we link the PWLR Lyapunov functions introduced in \cite{MA_cdc13} with results known in literature for piecewise linear Lyapunov functions \cite{molchanov86,polanski95,blanchini95}. Furthermore, we interpret our results in terms of contraction analysis and variational dynamics. The existence of Lyapunov function has also strong algebraic and dynamical implications. Particularly, results on the persistence and the uniqueness of equilibria are also established.

 This paper is organized as follows. In Section 2, we present the background and assumptions. Section 3 defines what we mean by a Robust Lyapunov function, and develops the uncertain systems framework. PWLR Lyapunov functions are introduced in Section 4, uniqueness of equilibria and persistence are presented in Section 5. Section 6 discusses the relationship with contraction analysis. Proofs are collected in the Appendix. \\ Note that some of the results in Section 3 and 4 in this paper have been published preliminarily without proofs in \cite{MA_cdc14}.

\paragraph*{Notation} Let $A \subset \mathbb R^n$ be a set, then $A^\circ, \bar A, \partial A, \co A$ denote its interior, closure, boundary, and convex hull, respectively. Given $x \in \mathbb R^n$, a column vector, its $\ell_\infty$-norm is $\|x\|_{\infty} = \max_{1\le i \le n} |x_i|$. 
The inequalities $x\ge0,\, x>0,\, x\gg 0$ denote elementwise nonnegativity, elementwise nonnegativity with at least one positive element, and elementwise positivity, respectively. $A$ is a signature matrix if it is diagonal and the diagonal entries belong to $\{\pm 1\}$.  For $A \in \mathbb R^{n \times \nu}$, $\ker(A)$ denotes the kernel or null-space of $A$, while $\im (A)$ denotes the image space of $A$.  $A \in \mathbb R^{n \times n}$ is Metzler if all off-diagonal elements are nonnegative.
The set of $n \times n$ real symmetric matrices is denoted by $\mathbb S^{n}$. Let $A \in \mathbb S^{n}$, then $A \ge \!(>) 0$ denotes $A$ being positive semi-definite (definite), respectively. $A \succeq 0$ denotes  entrywise nonnegativity.
 The all-ones vector is denoted by $\mathbf 1$, where its dimension can normally be inferred from the context. Let $\{A_i\}_{i=1}^k \subset \mathbb R^{n \times m}$, its conic hull denotes the set $\{ \sum_{i=1}^k\lambda_i A_i: \lambda_i \in \bar{\mathbb R}_+ \}$.
Let $V:D \to \mathbb R$, then the kernel of $V$ is $\ker(V)=V^{-1}(0)$. \textbf{\textit{T}}$M$  is the tangent bundle of the differentiable manifold $M$.

 \section{Background on Chemical Reaction Networks}
The field of CRN dynamics has an established literature \cite{feinberg95,angeli09tut}. We review here the relevant background.
  A reaction network has two mathematical defining features: the \emph{stoichiometry} and the \emph{kinetics}. Informally speaking, stoichiometry describes the relative number of molecules of reactants and products involved whenever each reaction occurs, while kinetics is concerned with the relations that govern the velocity of transformation of reactants into products. We explain both below.
\subsection{Stoichiometry}
A Chemical Reaction Network (CRN) is defined by a set of species $\mathscr S=\{X_1,..,X_n\}$, and a set of reactions $\mathscr R=\{\R_1,...,\R_\nu\}$. Each reaction is denoted as:
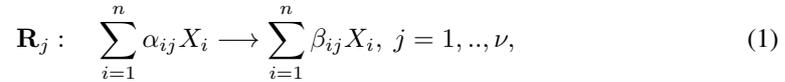
\begin{equation}\label{e.reaction}
    \R_j: \quad \sum_{i=1}^n \alpha_{ij} X_i \longrightarrow \sum_{i=1}^n \beta_{ij} X_i, \ j=1,..,\nu,
\end{equation}
where $\alpha_{ij}, \beta_{ij}$ are nonnegative integers called \emph{stoichiometry coefficients}. The expression on the left-hand side is called the \emph{reactant complex}, while the one on the right-hand side is called the \emph{product complex.} The forward arrow refers to the idea that the transformation of reactants into products is only occurring in the direction of the arrow. 
In order to allow external inflows or outflow the reactant or product complex can both be empty, though not simultaneously. 

The stoichiometry of a network can be summarized by arranging the coefficients in an augmented matrix $ n \times 2\nu$ as:
\begin{equation}\label{e.AB} \tilde \Gamma = [ A | B] , \mbox{where \,} [A]_{ij}= \alpha_{ij}, [B]_{ij}= \beta_{ij}. \end{equation}

 The two matrices can be subtracted to yield an $n \times \nu$ matrix $\Gamma=[\gamma_1^T \ .. \gamma_n^T]^T$ called the \emph{stoichiometry matrix}, which is defined as $\Gamma=B-A$, or element-wise as:
\[    [\Gamma] _{ij} = \beta_{ij}- \alpha_{ij}.\]

A left null vector $d \in \mathbb R^n, d^T \Gamma=0$ with $d>0$ is said to be a \emph{conservation law}. If there exists a conservation law $d\gg 0$, the network is said to be \emph{conservative}. \\
A reaction $\mathbf R_j$ is said to be an \emph{input reaction} to $X_i$ if $\beta_{ij}>0$, and is said to be an \emph{output reaction} of $X_i$ if $\alpha_{ij}>0$.
Let $P \in \mathscr S$ be a non-empty set of species. Denote the set of output reactions of species in $P$ by $\Lambda(P)$. 
Then, a nonempty set $P \subset  \mathscr S$ is called a \emph{siphon} \cite{Angeli07} if each input reaction
to a species in $P$ is also an output reaction of species in $P$. A siphon is a \emph{deadlock} if $\Lambda(P)=\mathscr R$. A siphon or a deadlock is said to be \emph{critical} if it does not contain a set of species corresponding to the support of a conservation law.

\subsection{Kinetics}
Assume we have an isothermal well-stirred chemical reactor; this implies that the species are distributed uniformly in the reactor. In order to study kinetics, a nonnegative number $x_i $ is associated to each species $X_i$ to denote its \emph{concentration}. Assume that the chemical reaction $\R_j$ takes place continuously in time. A \emph{reaction rate} or velocity function $R_j:\Rnn \to \bar{\mathbb R}_+$ is assigned to each reaction.

A widely-used expression which originates from statistical thermodynamics is given as:
\begin{equation}\label{e.mass_action}
    R_j(x)= k_j \prod_{i=1}^n x_i^{\alpha_{ij}},
\end{equation}
(the so called Mass-Action kinetics), with the convention $0^0=1$, where $k_j, j=1,..,m$ are positive numbers known as the kinetic constants, and are usually highly uncertain. \\
The Mass-Action model is not a universal model since there are other models which are popular in systems biology like the Michaelis-Menten model, and the Hill model.

 In this work we will pursue ''kinetics-independent'' approach. More precisely, we assume that the reaction rate function of a CRN is unknown except for satisfying the following assumptions:
 \begin{enumerate}
   \item[\bf AK1.] it is a  $\mathscr C^1$ function, i.e. continuously differentiable;
   \item[\bf AK2.] $x_i =0 \Rightarrow R_j(x)=0$, for all $i$ and $j$ such that $\alpha_{ij} > 0$; 
   \item[\bf AK3.] it is nondecreasing with respect to its reactants, i.e
\begin{equation}\label{e.rate_mono}
    \frac{\partial R_j}{\partial x_i}(x)  \left \{\begin{array}{ll} \geq 0 & : \alpha_{ij} > 0 \\  =0 &: \alpha_{ij}=0 \end{array} \right . .
\end{equation}
\item[\bf AK4.] The inequality in \eqref{e.rate_mono} holds strictly for all $x \in \mathbb R_+^n$.
 \end{enumerate}

 Reaction rate functions satisfying AK1-AK4 are called \emph{admissible}. For a given stoichiometric matrices $A,B$, the set of admissible reactions is denoted by $\mathscr K_A$. A network family is the triple $\mathscr N_{A,B}=(\mathscr S, \mathscr R,\mathscr K_A)$.

\begin{remark}
 The assumptions above imply the monotonic dependence of the reaction rate on the concentration of its reactants. This captures the basic intuition about the nature of a reaction since, as the concentration of reactants increases, the likelihood of collision between molecules increases, and hence the rate of the reaction. Note that, in principle, a monotonically decreasing dependence on the reactants can also be accommodated to model \emph{inhibition}.
\end{remark}
\begin{remark}\label{rem.sign} It can be noted that if a reaction rate function is admissible, then the corresponding Jacobian $\partial R/\partial x$  exhibits a certain zero sign-pattern that can be read from the graph, and any matrix satisfying that pattern correspond to an admissible $R$.
\end{remark}

\subsection{Dynamics}
The dynamics of a CRN with $n$ species and $\nu$ reactions are described by a system of ordinary differential equations (ODEs) as:
\begin{equation}\label{e.ode}
    \dot x (t) = \Gamma R(x(t)), \ x(0) \in \Rnn
\end{equation}
where $x(t)$ is the concentration vector evolving in the nonnegative orthant  $\bar{\mathbb R}_+^n$, $\Gamma \in \mathbb R^{n \times \nu}$ is the stoichiometry matrix, $R(x(t))=[R_1(x(t)), R_2(x(t)), ..., R_\nu(x(t))]^T \in \bar{\mathbb R}_+^\nu$ is the reaction rates vector.

Note that \eqref{e.ode} belongs to the class of \emph{positive systems}, i.e, $\Rnn$ is forward invariant. In addition, the manifold $\mathscr C_{x_\circ}:=(\{x_\circ\}+ \mbox{Im}(\Gamma)) \cap \Rnn$ is forward invariant, and it is called \emph{the stoichiometric compatibility class} associated with $x_\circ$. Therefore, all stability results in this paper are relative to the stoichiometry compatibility class. Note that for a conservative network all stoichiometric classes are compact convex polyhedral sets.

 An equilibrium $x_e$ of \eqref{e.ode} is \emph{non-degenerate} if the Jacobian evaluated at $x_e$ relative to $\mathscr C_{x_e}$ is nonsingular. 
More precisely, considering changing coordinates using a transformation matrix $T= [ T_1^T  \ D ]^T,$
where $D^T$ has full row rank and $D^T \Gamma =0$, and $T_1$ is any matrix such that $T$ is nonsingular. Then, the Jacobian in the new coordinates can be written as:
\begin{align}\label{e.reduced_jacobian} T \Gamma \frac{\partial R}{\partial x} T^{-1} = \begin{bmatrix} J_1 & J_2 \\ 0 & 0 \end{bmatrix}. \end{align}
Therefore, $x_e$ is nondegenerate iff $J_1$ evaluated at $x_e$ is nonsingular. $J_1$ is called a \emph{reduced Jacobian}.

 The stoichiometry of the network will be assumed to satisfy the following assumption:
\begin{enumerate}
\item[\text {\textbf{AS}}] There exists $v \in \ker \Gamma$ such that $v\gg0$. This condition is necessary for the existence of a steady state in which all concentrations are positive.
\end{enumerate}

\section{Robust Lyapunov Functions and Linear Differential Inclusions}

%
\subsection{Robust Lyapunov Functions}
In order for the stability analysis of CRNs to be independent of the specific kinetics, we aim at constructing Lyapunov functions which are dependent only on the graphical structure, and hence are valid for all reaction rate functions that belongs to $\mathscr K_A$. Therefore, we state the following definition.

\begin{definition}[Robust Lyapunov Function]\label{def.rlf} Consider \eqref{e.ode} and let $x_e$ be an equilibrium. Let $\tV:\bar{\mathbb R}^q \to \bar{\mathbb R}_+$ be locally Lipschitz, and let $W_{R,x_e}:{\mathbb R}^n \to \mathbb R^q$ be a $\mathscr C^1$ function. Then, $(\tV,W_{R,x_e})$ is said to induce a \emph {Robust Lyapunov Function} (RLF) with respect to the network family $\mathscr N_{A,B}$ if for any choice of $R \in \mathscr K_A, x_e \in \Rnn$, the function $V_{R,x_e}=\tV \circ W_{R,x_e}$ is
\begin{enumerate}
 \item \emph {Positive-Definite}: $V_{R,x_e}(x)\ge0$, and $V_{R,x_e}(x)=0$ if and only if $R(x) \in \ker \Gamma$.
\item \emph {Nonincreasing}: $\dot V_{R,x_e} (x) \le 0$ for all $x \in \mathscr C_{x_e}$.
\end{enumerate}
A network for which an RLF exists is termed a {\bfseries Graphically Stable Network} (GSN).
\end{definition}
\begin{remark}As will be seen later, the function $\tV$ used in the definition of the Lyapunov function (through composition) is invariant with respect to the specific network realization in $\mathscr K_A$, while the function $W_{R,x_e}$ is allowed to depend on the kinetics of the network. Two main examples of the function $W_{R,x_e}$ are $W_{R,x_e}(x)=R(x)$, and $W_{R,x_e}(x)=x-x_e$.
With a mild abuse of terminology, we call the parameterized Lyapunov function $V_{R,x_e}$ an RLF.
\end{remark}
\begin{remark}The time-derivative in the definition above is the upper right Dini's derivative \cite{yoshizawa}:
\begin{equation}\label{e.dini} \dot V (x) := \limsup_{h \to 0^+} \frac{V(x+h \Gamma R(x) ) - V(x)}h, \end{equation}
 which is finite for all $x$ since $V$ is locally Lipschitz.
\end{remark}

Since the RLF defined above is not strict, we need the following definition.
\begin{definition}[The LaSalle's Condition]\label{def.lasalle_rlf}
An RLF $V_{R,x_e}$ for $\mathscr N_{A,B}$ is said to satisfy \emph{the LaSalle's Condition} if for any choice $R \in \mathscr K_A$ the following statement holds. \\ If a solution $\varphi(t;x_\circ)$ of \eqref{e.ode} satisfies $ \varphi(t;x_\circ) \in \ker \dot V \cap \mathscr C_{x_e}$, $t \ge 0$, then this implies that $\varphi(t;x_\circ) \in E_{x_\circ}$ for all $t \ge 0$, where $E_{x_\circ} \subset \mathscr C_{x_\circ}$ is the set of equilibria for \eqref{e.ode} contained in  $\mathscr C_{x_\circ}$.\end{definition}

The following theorem adapts Lyapunov's second method \cite{yoshizawa,PWLRj} to our context.
\begin{theorem}[Lyapunov's Second Method]\label{th.lyap_rlf}
  Given \eqref{e.ode} with initial condition $x_\circ \in \mathbb R_+^n$, and let $\mathscr C_{x_\circ}$ as the associated stoichiometric compatibility class. Assume there exists an RLF Lyapunov function and suppose that $x(t)$ is bounded,
  \begin{enumerate}
  \item Then the equilibrium set $E_{x_\circ}$ is Lyapunov stable.
   \item If, in addition, $V$ satisfies the LaSalle's Condition, then
   $x(t) \to E_{x_\circ}$ as $t \to \infty$ (i.e., the point to set distance of $x(t)$ to $E_{x_\circ}$ tends to $0$). Furthermore, any isolated equilibrium relative to $\mathscr C_{x_\circ}$ is asymptotically stable.
   \item  If $V$ satisfies the LaSalle's Condition, and all the trajectories are bounded, then: if there exists $x^* \in E_{x_\circ}$, which is isolated relative to $\mathscr C_{x_\circ}$ then it is unique, i.e., $E_{x_\circ}=\{x^*\}$. Furthermore, it is a globally asymptotically stable equilibrium relative to $\mathscr C_{x_\circ}$.
   \end{enumerate}
 \end{theorem}

\begin{remark} Note that the RLF considered can not be used to establish boundedness of solutions, as it may fail to be proper.  Therefore, we need to resort to other methods so that boundedness can be guaranteed. For instance, if the network is conservative, i.e the exists $w \in \mathbb R_+^n$ such that $w^T\Gamma =0$, which ensures the compactness of $\mathscr C_{x_\circ}$.
\end{remark}

\subsection{Uncertain Systems Framework: Reaction Coordinates}
In this subsection  the function $W_{R,x_e}$ is assigned to be $R$, as in the case of PWLR functions.
As arbitrary monotone kinetics are allowed in our formulation of the CRN family $\mathscr N_{A,B}$, the system \eqref{e.ode} with kinetics $\mathscr K_A$ can be viewed as an uncertain system. However, this system is not in the form of the traditional types of parameter uncertainties treated in the literature. In this subsection, we show that shifting the analysis of the system to reaction coordinates enables  to view it as a \emph{linear parameter varying} (LPV) system for which existence of a common Lyapunov function directly yields a robust Lyapunov function for the original CRN.

Let $r(t):=R(x(t))$, then we have:
\begin{equation}\label{e.Rode}
  \dot r(t) = \frac{\partial R}{\partial x}(x(t)) \Gamma r(t) = \rho(t) \Gamma r(t),
\end{equation}
   where $\rho(t):=\frac{\partial R}{\partial x}(x(t))$. We can write $\rho (t)$ as a conic combination of individual partial derivatives as follows:
   \begin{equation}\label{e.rho_decomp}
  \frac{\partial R}{\partial x}(x(t))=\rho (t) = \sum_{i,j:\alpha_{ij}>0} \rho_{ji}(t) E_{ji},
\end{equation}
where $[\rho(t)]_{ji}=\rho_{ji}(t)$, and $[E_{ji}]_{j'i'}=1$ if $(j',i')=(j,i)$ and zero otherwise.\\ Let $s$ denote the number of elements in the support of $\partial R/\partial x$, and let $\kappa:\{1,..,s\} \to \{ (i,j):\alpha_{ij}>0 \}$ be an indexing map.  Then, we can write \eqref{e.Rode} as:
\begin{equation}\label{e.lpv}
  \dot r = \sum_{i,j:\alpha_{ij}>0} \rho_{ji}(t) E_{ji} \Gamma r=\sum_{\ell=1}^s \rho_\ell(t) \Gamma^{\ell}  r,
\end{equation}
where $\Gamma^{\ell}=e_j\gamma_i^T $, $\rho_{\ell}(t) = \rho_{ji} (t) $, with $(i,j)=\kappa(\ell)$, and $\{e_j\}_{j=1}^\nu$ denotes the canonical basis of $\mathbb R^\nu$.  Hence, equation \eqref{e.lpv} represents a linear parameter-varying system which has $s$ nonnegative time-varying parameters $\{\rho_1(t),..,\rho_s(t)\}$ and the system matrix belongs to the conic hull of the set of rank-one matrices $\{\Gamma^{1},...,\Gamma^s \}$.

 Hence, we have the following definition.
 \begin{definition}[Common Lyapunov Function]\label{def.commonLyap}A function $\tV: \bar{\mathbb R}_+^\nu \to  \bar{\mathbb R}_+$ is said to be a Lyapunov function for the linear system $\dot r =\Gamma ^\ell r$ if it is locally Lipschitz, nonnegative, has a negative semi-definite time-derivative along the trajectories of the linear system, and $\ker \tV \subset \ker \Gamma^\ell$.  Furthermore, $\tV$ is said to be a \emph{common Lyapunov function} for the set of linear systems $\{\dot r=\Gamma^{1}  r,...,\dot r=\Gamma^{s}  r \}$ if it is a Lyapunov function for each of them, and $\ker \tV = \bigcap_{\ell=1}^s \ker \Gamma^{\ell}$.
 \end{definition}

 Equipped with the above definitions, we are ready to state the main result for this Section.  Its proof is deferred to the appendix for the sake of readability.

\begin{theorem}[Equivalence btw. CLF and RLF in reaction coordinates]\label{th.uncertain} Given the system \eqref{e.ode}. There exists a common Lyapunov function $\tV: \bar{\mathbb R}_+^\nu \to  \bar{\mathbb R}_+$ for the set of linear systems $\{\dot r=\Gamma^{1}  r,...,\dot r=\Gamma^{s}  r \}$ if and only if $(\tV,R)$ induces the Robust Lyapunov function parameterized as $V_R(x)=\tV(R(x))$ for the CRN family $\mathscr N_{A,B}$.
\end{theorem}

\begin{remark} Since the zero matrix belongs to the conic hull of $\{\Gamma^1,...,\Gamma^s\}$, asymptotic stability can't be established by the mere existence of the common Lyapunov function. A LaSalle's argument is needed as will be mentioned in the following Section.
\end{remark}

\begin{remark}Although the dynamics of concentrations \eqref{e.ode} or the dynamics of the extent of reaction \eqref{e.extent} define positive systems, the differential linear inclusion that can be defined from \eqref{e.Rode} is not a positive linear differential inclusion. More precisely, this means that it is not necessarily true that all the matrices that belong to conic hull $\{\Gamma^{1},...,\Gamma^{s}\}$ are Metzler.
\end{remark}

\subsection{Dual Robust Lyapunov Function: Species Coordinates}
The RLF introduced in the previous subsection is a function of $R(x)$. We investigate now RLFs that are functions of the difference $x-x_e$. This is carried out in a manner that is dual to what has been done in the previous subsection.

In order to present the dual framework, an alternative representation of the system dynamics can be adopted. Consider a CRN as in \eqref{e.ode}, and let $x_e$ be an equilibrium.  Then, there exists $x{''}(x) \in \bar{\mathbb{R}}_+^n$ such that \eqref{e.ode} can written equivalently as:
\begin{equation}\label{e.alternative_representation}\dot x= \Gamma \frac{\partial R}{\partial x}(x{''}) (x-x_e),   x(0) \in \mathscr C_{x_e} \end{equation}

The existence of $x'':=x_e+ \varepsilon_x (x-x_e)$  for some $\varepsilon_x \in [0,1]$ follows by applying the Mean-Value Theorem to $R(x)$ along the segment joining  $x_e$ and $x$.

Let $z=x-x_e$, then similar to the previous section, the conic combination \eqref{e.rho_decomp} can be used to rewrite \eqref{e.alternative_representation} as:
\begin{equation}\label{e.rho_dual}\dot z= \Gamma \frac{\partial R}{\partial x}(x{''}) (x-x_e)= \sum_{\ell=1}^s \rho_{\ell}(t) \Gamma E^{\ell} z = \sum_{\ell=1}^s \rho_{\ell}(t) \Gamma_{i_\ell} e_{j_\ell}^T z, \end{equation}
where $\rho_\ell(t)=\frac{\partial R_{j_\ell}}{\partial x_{i_\ell}}(x''(x(t))$, and $\Gamma_i$ is the $i^{\mathrm{th}}$ column of $\Gamma$. Therefore, the system dynamics has been embedded in the linear differential inclusion with vertices $\{\Gamma_{i_1}e_{j_1}^T , ..., \Gamma_{i_s} e_{j_s}^T\}$.

Let $D^T$ be a matrix with columns that are the basis vectors of $\ker \Gamma^T$. The following theorem can be stated.
\begin{theorem}[Equivalence CLF and RLF, species coordinates]\label{th.uncertain_dual} Given the system \eqref{e.ode}. There exists a common Lyapunov function $\hV: {\mathbb R}^n \to  \bar{\mathbb R}_+$ for the set of linear systems $\{\dot z= (\Gamma_{i_1} e_{j_1}^T) z , ..., \dot z= (\Gamma_{i_s} e_{j_s}^T) z   \}$, on the invariant subspace $\{z: D^T z=0\}$ if and only if $(\tV,W_{x_e}),W_{x_e}=x-x_e$ induces the Robust Lyapunov function parameterized as $V_{x_e}(x)=\hV(x-x_e)$ for the CRN family $\mathscr N_{A,B}$.
\end{theorem}
\begin{proof}
Let $V(x)=\hV(x - x_e)$, then when $\partial \hV/\partial z$ exists we can write:
\begin{align*} \dot V &= \frac{\partial \hV}{\partial z} \dot z= \frac{\partial \hV}{\partial z} \sum_{\ell=1}^s \rho_{\ell}(t) \Gamma_{i_\ell} e_{j_\ell}^T z =  \sum_{\ell=1}^s \rho_{\ell}(t) \left ( \frac{\partial \hV}{\partial z}\Gamma_{i_\ell} e_{j_\ell}^T z \right ). \end{align*}
Since we have assumed that $\hV$ is a common Lyapunov function for the set of linear systems $\{\dot z= (\Gamma_{i_1} e_{j_1}^T) z , ..., \dot z= (\Gamma_{i_s} e_{j_s}^T) z   \}$ the proof can proceed in both directions in a similar way to the proof of Theorem \ref{th.uncertain}. Notice that the constraint $D^T z =0$ is needed since $D^T \dot x(t) \equiv 0$ is implicit in the structure of the original system \eqref{e.ode}.
\end{proof}


\subsection{Relationship Between the Two Frameworks}

 We show next that if $\tV$ used in reaction coordinates satisfies a relatively mild additional assumption, then the Lyapunov function of the form $\hV(x-x_e)$ can be used, where $x_e$ is an equilibrium point for \eqref{e.ode}.

To this end, the following theorem can be stated.
\begin{theorem}[Converting Lyapunov functions between reaction and species coordinates]\label{th.dual_uncertain} Let $V_1(x)=\tV(R(x))$ represent an RLF for the network family $\mathscr N_{A,B}$. If there exists $\hV: \mathbb{R}^n \to \bar{\mathbb R}_+$ such that for all
 $r \in \mathbb{R}^{\nu}$:
\color{black}
\begin{equation}\label{e.two_lyap}
\tV(r)=\hV(\Gamma r),
\end{equation}
then $V_2(x)=\hV(x-x_e)$ represents an RLF for the network family $\mathscr N_{A,B}$, where $x_e$ is any equilibrium point for \eqref{e.ode}.
\end{theorem}
\begin{proof} Condition 1 in Definition \ref{def.rlf} is clearly satisfied. It remains to show the second condition. Let $z=x - x_e$. Then, whenever $\hV$ is differentiable:
\[ \dot{V_2}(x)= \frac{\partial \hV(x-x_e)}{\partial z} \dot x = \frac{\partial \hV(x-x_e)}{\partial z} \Gamma R(x), \]
Before proceeding, we prove two statements: First, from \eqref{e.two_lyap}, we get $(\partial \tV(r) /\partial r)= (\partial \hV(\Gamma r) /\partial z) \Gamma$. Second, note that $x-x_e \in \im(\Gamma)$, hence there exists $R \in \mathbb{R}^{\nu}$ such that $\Gamma R = x - x_e$, where $R$ can always be chosen nonnegative by assumption AS.  Hence, where $\hV$ is differentiable, we can use \eqref{e.alternative_representation} to write:
\begin{align*} \dot{V_2}(x) &=  \frac{\partial \hV(x-x_e)}{\partial r} \Gamma  \frac{\partial R(x'')}{\partial x} (x-x_e)
 = \frac{\partial \tV(R)}{\partial r}   \frac{\partial R(x'')}{\partial x} \Gamma R \\ & = \sum_{\ell=1}^s \rho_{\ell} \frac{\partial \tV(R)}{\partial r} \Gamma^\ell R \le 0, \end{align*}
where the last inequality follows from \eqref{e.uncertain_proof}. Lemma \ref{lem.lip} implies that $\dot V_2(x) \le 0$ for all $x$. \end{proof}

\begin{remark}\label{dual_lyap_weakness} $V_2$ has a simpler structure than $V_1$ since it depends on $x-x_e$. However, it can be noted in the proof that for each specific choice of $x_e$, the Lyapunov function $\hat V(x-x_e)$ is nonincreasing only along solutions contained in $\mathscr C_{x_e}$.
\end{remark}

\color{black}
\paragraph{\bf Relationship to the Extent of Reaction Formulation} \strut\\
Recall that the extent of reaction \cite{clarke80} is defined as: $\xi(t)=\int_{0}^t R(x(\tau)) d\tau + \xi(0)$.
 If $x(t) \in \mathscr C_{x_e}$, then $\exists \xi^* \ge 0$ such that $x_e-x_\circ=\Gamma \xi^*$. We set $\xi(0):= \xi^*$. Hence, we can write:
\begin{equation}\label{e.extent}
\Gamma\xi(t) = x(t)-x_e,
\end{equation}
and
\begin{equation}\label{e.extent_ode_e}
\dot\xi = R(x_e + \Gamma \xi ), \xi(0):= \xi^*,
\end{equation}
which is the extent-of-reaction ODE representation of the dynamics of the CRN.

Therefore, we state the following result.

\begin{corollary}\label{th.four_lyap} Let $\Gamma$ be given and $\tV: \bar{\mathbb R}_+^\nu \to \bar{\mathbb R}_+$ be a function. Assume there exists $\hV: \mathbb R_+^n \to \bar{\mathbb R}_+$ such that for all 
$r \in \mathbb{R}^{\nu}$, \color{black}
$\tV(r)=\hV(\Gamma r)$. If $\tV$ is a common Lyapunov function for $\{\dot r= e_{j_1} \gamma_{i_1}^T r, ..., \dot r= e_{j_s} \gamma_{i_s}^T r  \}$ where $(i_\ell,j_\ell)=\kappa(\ell)$ for all $\ell \in \{1,2, \ldots, s \}$, then
   $\tV(\xi)$ is nonnegative, and nonincreasing along the trajectories of $\dot \xi = R(x_e+\Gamma \xi)$ for any $R \in \mathscr K_{A}$.
\end{corollary}

\begin{proof} The first two statements follow from Theorems \ref{th.uncertain} and \ref{th.dual_uncertain}. We prove now the third statement. Using \eqref{e.two_lyap}, \eqref{e.extent} we get $\tV(\xi)=\hV(x-x_e)$. Therefore, the required statement follows from the result that $\hV(x-x_e)$ is nonincreasing along the trajectories of \eqref{e.ode}.
\end{proof}

\section{Application to PWLR Lyapunov Functions}
\subsection{Relationship to Previous Results}

In the previous papers \cite{MA_cdc13,PWLRj} the concept of Piecewise Linear in Rate (PWLR) Lyapunov functions has been introduced based on a direct analysis of the CRN. Such functions satisfy the conditions of Definition \ref{def.rlf}, and hence they are Robust Lyapunov functions. In this section we show that those results can be interpreted in the uncertain systems framework introduced above. This also allows to provide alternative algorithms for the existence and construction of PWLR functions.

Consider a CRN \eqref{e.ode} with a $\Gamma \in \mathbb R^{n \times r}$. Two representation of the PWLR Lyapunov function have been discussed. Given a partitioning matrix $H \in \mathbb R^{p \times r}$ such that $\ker H=\ker \Gamma$. PWLR Lyapunov functions are piecewise linear in rates, i.e., they have the form: $V(x)=\tV(R(x))$, where $\tV: \mathbb R^\nu \to \mathbb R$  is a continuous PWL function given as 
\begin{equation*}\label{e.conLF}
 \tilde V(r) = |c_k^T r|, \ r \in \pm\cW_k , k=1,..,m/2,
\end{equation*}
where the regions $\cW_k = \{r \in \mathbb R^\nu: \Sigma_k Hr \ge 0 \},k=1,..,m$ form a proper conic partition of $\mathbb R^\nu$, while $\{\Sigma_{k}\}_{k=1}^m$ are signature matrices with the property $\Sigma_{k}=-\Sigma_{m+1-k}, k=1,..,m/2$.
  The coefficient vectors of each linear component can be collected in a matrix $C=[c_1,..,c_{\frac m2} ]^T \in \mathbb R^{\frac m2 \times r}$.  If the function $\tV$ is convex, then we have the following simplified representation of $V$:
  \[ V(x) = \| C R(x) \|_{\infty}. \]
  This representation reminds of the $\ell_\infty$-norm Lyapunov functions that have been used for linear systems in \cite{polanski95}. In fact, the next theorem establishes the link between the results introduced in \cite{PWLRj}
	\color{black} for checking candidate PWLR functions based on direct analysis and previous work on $\ell_\infty$ Lyapunov functions using the framework introduced in the previous section.

\begin{proposition}\label{th.pwlr_check} Given $\Gamma$ and $H$. Let $V=\tV \circ R$ be a candidate continuous nonnegative PWLR with $C=[c_1 \ ... \ c_{\frac m2}]^T\in \mathbb R^{\frac m2 \times r}$. Then $(\tV,R)$ induces an RLF if and only if:
\begin{enumerate}
  \item $\ker C = \ker \Gamma$, and
  \item there exists matrices $\{\Lambda^{\ell}\}_{\ell=1}^s \subset \mathbb R^{\frac m2 \times {\frac m2}}$ such that
  \begin{equation}\label{e.NewCheck_con}
    \Lambda^{\ell} H = -C \Gamma^{\ell},
  \end{equation}
  and $\lambda_k^{\ell} \Sigma_k >0$, where $\Lambda^{\ell}=[{\lambda_1^{\ell}}^T ... {\lambda_{m/2}^{\ell}}^T]^T$.
   \end{enumerate}

If $\tV$ is convex, then the second condition can be replaced with
\begin{enumerate}
  \item[2)] there exists  Metzler matrices $\{\Lambda^{\ell}\}_{\ell=1}^s \subset \mathbb R^{m \times m}$ such that
    \begin{equation}\label{e.NewCheck}
    \Lambda^{\ell} \tilde C = \tilde C \Gamma^{\ell},
  \end{equation}
  and  $\Lambda^{\ell} \mathbf 1 =0$ for all $\ell=1,..,s$, where $\tilde C=[C^T \ -\! C^T]^T$.
 \end{enumerate}
\end{proposition}
\begin{proof} The proof can be carried out by performing elementary algebraic manipulations on the results presented in \cite[Theorems 4,5]{PWLRj}. The details are omitted for the sake of space.
\end{proof}
\begin{remark} The symmetries in equation \eqref{e.NewCheck} imply that it can be written equivalently as:
\begin{equation}\label{e.NewCheck_tilde} C \Gamma^{\ell}= \tilde\Lambda^{\ell} C,\end{equation}
where $\tilde\Lambda^{\ell}$ is an $\tfrac m2 \times \tfrac m2$ matrix which is defined by subtracting the upper $\tfrac m2 \times \tfrac m2$ blocks of $\Lambda^{\ell}$ from each other.  $\tilde\Lambda^{\ell}$ satisfies:
\begin{equation} \label{e.lognorm_expression}\max_{k} \left ( \tilde\lambda_{kk}^{(\ell)} + \sum_{j \ne k} | \tilde\lambda_{kj}^{(\ell)} | \right ) \le 0.\end{equation}
This is exactly the condition that $\ell_\infty$-norm Lyapunov functions need to satisfy
  for a linear system \cite{molchanov86,kiendl92}. This shows that Theorem \ref{th.uncertain} provides the framework to utilize the existing linear stability analysis techniques in the literature to construct robust Lyapunov functions for nonlinear systems such as CRNs. For example, we can verify $\ell_1$ Lyapunov functions of the form $V(x)=\|CR(x)\|_{1}$ directly by replacing condition \eqref{e.lognorm_expression} by
  \begin{equation}  \max_{k} \left ( \tilde\lambda_{kk}^{(\ell)} + \sum_{j \ne k} | \tilde\lambda_{jk}^{(\ell)} | \right ) \le 0,\end{equation}
  instead of converting them to the $\ell_\infty$-norm form.
\end{remark}

Three construction algorithms have presented in \cite{PWLRj} and we revisit the second one here.
Before proceeding to the construction algorithm, we need to introduce the concept of a neighbor to a region. Fix $k\in\{1,..,m/2\}$. Consider $H$: for any pair of linearly dependent rows $h_{i_1}^T,h_{i_2}^T$ eliminate $h_{i_2}^T$. Denote the resulting matrix by $\tilde H \in \mathbb R^{\tilde p \times \nu}$, and let $\tilde\Sigma_1,..,\tilde\Sigma_m$ the corresponding signature matrices. Therefore, the region can be represented as $\cW_k=\{r|\tilde \Sigma_k \tilde Hr \ge 0\}$. The \textit{distance} $d_r$ between two regions $\cW_k,\cW_j$ is defined to be the \textit{Hamming distance} between $\tilde\Sigma_{k},\tilde\Sigma_j$. Hence, the set of neighbors of a region $\cW_k$
are defined as:
\begin{align*}
 \mathcal N_k &= \{ j \in \{1,2,\ldots,m \}: d_r(\cW_j,\cW_k)=1 \}, 
\end{align*}
Equivalently, note that a neighboring region to $\cW_k$ is one which differs only by the switching of one inequality. Denote the index of the switched inequality by the map $s_{k}(.):\mathcal N_k \to \{1,..,p\}$. For simplicity, we use the notation $s_{k\ell}:=s_k(\ell)$.

We use Theorem \ref{th.uncertain} to show that the problem of constructing a PWLR Lyapunov function over a given partition, i.e. a given $H$, can be solved via linear programming. However, instead of encoding the nondecreasingness condition into precomputed sign patterns as in the previous chapter, we use here alternative conditions which are stated in the following proposition.

\begin{proposition}\label{th.pwlr_c}Given the system \eqref{e.ode} and a partitioning matrix $H \in \mathbb R^{p \times r}$. Consider the linear program:
\begin{equation*}
\begin{aligned} 
&  {\text{Find}}
& & c_k, \xi_k, \zeta_k \in \mathbb R^\nu, \Lambda^{\ell} \in \mathbb R^{m \times m}, \eta_{kj} \in \mathbb R, \\ &&& k=1,..,\tfrac m2; j \in \mathcal N_k, \ell=1,..,s,\\
& \text{subject to}
& & c_k^T=\xi_k^T \Sigma_k H ,\\ &&&
  C \Gamma^{\ell} =- \Lambda^{\ell} H,  \lambda_k^\ell \Sigma_k \ge 0, \\ &&&
 c_k-c_j=\eta_{kj}\sigma_{ks_{kj}}h_{s_{kj}}, \\ &&&
\xi_k \ge0,   \mathbf 1^T \xi_k>0,
\Lambda^{\ell} \ge0. 
\end{aligned}
\end{equation*}
Then there exists a PWLR RLF with partitioning matrix $H$ if and only if there exists a feasible solution to the above linear program that satisfies $\ker C=\ker \Gamma$ satisfied. Furthermore, the PWLR RLF can be made convex by adding the constraints $\eta_{kj}\ge 0$.
\strut
\end{proposition}
\begin{remark} A natural choice for $H$ is $H:=\Gamma$. The Lyapunov function reduces then to:
\[ V(x)=\|\diag(\xi_k) \dot x \|_1, \ R(x) \in \mathbb \cW_k,\]
which generalizes the Lyapunov function presented in \cite{maeda78}.
\end{remark}
\color{black}

\begin{remark}The LaSalle's Condition can be verified  via a graphical algorithm described in \S III-F in \cite{PWLRj}.
\end{remark}

\subsection{The Dual PWL Lyapunov Function}
In \S III-C it has been shown that if there exists $\hV$ such that $\tV(r)=\hV(\Gamma r)$, then there exists a dual RLF for the same network family. In the case of PWLR Lyapunov functions condition 1 in Proposition \ref{th.pwlr_check} implies that this condition is always fulfilled.  Hence, consider a PWLR Lyapunov function defined with a partitioning matrix $H$ as in \eqref{e.conLF}. By Proposition \ref{th.pwlr_check} and the assumption that $\ker H=\ker \Gamma$, there exists $G \in \mathbb R^{p \times n}$ and $B \in \mathbb R^{\frac m2 \times n}$ such that $H=G\Gamma$ and $C= B \Gamma$. Similar to $\{\cW\}_{k=1}^m$, we can define the regions:
\[ \mathcal V_k = \{ z |\Sigma_{k} G z \ge 0 \}, k=1,..,m, \]
where it can be seen that $\mathcal V_k$ has nonempty interior iff $\cW_k$ has nonempty interior.\\
Therefore, as the pair $(C,H)$ specify the PWLR function fully, also the pair $(B,G)$ specifies the function:
\[ \hV(z)= b_k^T z, \, \mbox{when} \, \Sigma_{k}G z\ge 0,\]
where $B=[b_1, ..., b_{\frac m2} ]^T$.
If $\tV$ is convex, then it can be written in the form: $V_1(x)=\|C R(x)\|_{\infty}$. Similarly, the convexity of $\hV$ implies that \[V_2(x)=\|B(x-x_e)\|_{\infty},\] where the latter is the Lyapunov function used in \cite{blanchini14}.

Theorem \ref{th.dual_uncertain} established that if $\tV(R(x))$ is an RLF, then $\hV(x-x_e)$ is an RLF also. The following theorem shows that converse holds also for PWLR RLFs, however, it is worth recalling from Remark \ref{dual_lyap_weakness} that $\tV(R(x))$ is nonincreasing for all initials conditions, while $\hV(x-x_e)$ is nonincreasing only on $\mathscr C_{x_e}$.

\begin{theorem}\label{th.pwlr_dual} Given \eqref{e.ode}. Then, if there exists $G \in \mathbb R^{p \times n}$ and $B \in \mathbb R^{\frac m2 \times n}$ such that:
\begin{enumerate}
\item $(B\Gamma,G\Gamma)$ defines a PWLR RLF, then $(B,G)$ defines a dual PWL RLF.
    \item $(B,G)$ defines a dual PWL RLF, then $(B\Gamma,G\Gamma)$ defines a PWLR RLF.
\end{enumerate}
\end{theorem}
The proof is presented in the appendix.

\begin{remark}  Since $ D^T(x-x_e)=0$ for $x \in \mathscr C_{x_e}$, then if $\|B(x-x_e)\|_{\infty}$ is an RLF, then $\|(B+ Y D^T)(x-x_e)\|_{\infty}$ is also an RLF for an arbitrary matrix $Y$. Furthermore, since Theorem \ref{th.pwlr_dual} has shown that the reaction-based and the species-based representations are equivalent; it is easier to check and construct RLFs in the reaction-based formulation and they hold the advantage of being decreasing over all stoichiometry classes.
\end{remark}

\section{Properties of Graphically Stable Networks}
\subsection{Robust Non-singularity }
It has been shown in \cite{PWLRj} that the Jacobian of any network admitting a PWLR RLF is $P_0$, which implies that all principal minors are nonnegative. This can be used to show the following result.
\begin{theorem}[Robust Nonsingularity]\label{th.robust_nonsingular} Given \eqref{e.ode}. Assume that it is a graphically stable network. If for some realization $R \in \mathscr K_A$ there exists a point in the interior of a proper stoichiometric class such that the reduced Jacobian is non-singular at it, then the reduced Jacobian is non-singular in the interior of $\mathbb R_+^n$ for any realization of the kinetics. This implies that any positive equilibrium of this network is isolated and non-degenerate relative to its class.
\end{theorem}
\begin{proof}
Recall that for a GSN the negative Jacobian is $P_0$ for any choice of $R \in \mathscr K_A$ \cite{PWLRj}.
Using the Cauchy-Binet formula \cite{banaji07}, let $I \subset \{1,..,n\}$ be an arbitrary subset so that $|I|=k$. The corresponding principal minor can be written as:
\[\det\strut_{I}\left (-\Gamma \frac{\partial R}{\partial x} \right ) = \sum_{J \subset \{1,..,\nu\}, |J|=k}  \det(-\Gamma_{IJ}) \det\left(\frac{\partial R}{\partial x}_{JI}\right ) \mathop{=}^{(*)} \sum_{\iota} a_\iota \prod_{\ell \in L_\iota \subset \{1,..,s\}} \rho_\ell,\]
where $(*)$ refers to the fact that the sum can be expressed as a linear combination of products of $\rho_1,...,\rho_s$. We claim that the coefficients $a_\iota$ are all nonnegative. To show this, assume for the sake of contradiction that there is some $a_{\iota_*}$  negative. If we set all $\rho$'s to zero except the ones appearing in the $\iota_*^{\mathrm{th}}$ term, then this implies that the corresponding principal minor can be negative; a contradiction.

Now, the theorem can be proven by noting that the reduced Jacobian is non-singular iff the sum of all $k \times k$ principal minors of the negative Jacobian is positive, where $k=\mbox{rank}(\Gamma)$. Since it  is assumed that there exists a point for which the reduced Jacobian is non-singular, this implies that the sum of principal minors is positive for some choice of $\rho_1,..,\rho_s$. Since all of the principal minors are nonnegative, then at least one of them is positive. By AK4, that principal minor stays positive for any choice of positive $\rho_1,..,\rho_s$, i.e. it stays positive over the interior of $\mathbb R_+^n$.
\end{proof}

\subsection{Uniqueness of Equilibria}
An important result links injectivity of a map to the notion of $P$-matrix, \cite{nikaido65}. This states that a map is injective if its Jacobian is a $P$ matrix. For an injective vector-field, if an equilibrium exists, then it is unique. This notion has been studied extensively for reaction networks, and specialized by means of appropriate graphical conditions, \cite{banaji07,banaji09}.

Hence, the following theorem follows:
\begin{theorem}[Uniqueness of Positive Equilibria of Graphically Stable Network]\label{cor.p0_app}
 If $\mathscr N_{A,B}$ is GS, then it can not admit multiple nondegenerate positive equilibria in a single stoichiometric compatibility class.
 Furthermore, if there exists an isolated non-degenerate positive equilibrium $x_e$, relative to  $\mathscr C_{x_e}$, then it is unique.
\end{theorem}
\begin{proof}
It has been established in \cite[Appendix B]{banaji09} that the network can not admit multiple nondegenerate positive equilibria in a single stoichiometric compatibility class if the Jacobian is $P_0$. Hence, the first statement follows.

For the second statement, Theorem \ref{th.robust_nonsingular} has shown that the the existence of an non-degenerate positive equilibrium $x_e$ ensures that the reduced Jacobian  is non-singular on the interior of the orthant. In order to show uniqueness, assume for the sake of contradiction that there exists $y \ne x_e, y \in \mathscr C_{x_e}$ such that $\Gamma R(y) =0$. Then the fundamental theorem of calculus implies,
\[ 0=\Gamma R(x_e) - \Gamma R(y) \mathop{=}  \Gamma \int_0^1 \frac{\partial R}{\partial x}(tx_{e}+(1-t)y) \, (x_e- y) dt \mathop{=} \Gamma  \frac{\partial R}{\partial x}(x^*) (x_e-y),   \]
where $x^*= t^* x_e + (1-t^*) y$, and $t^* \in (0,1)$. The existence of $t^*$ is implied by the integral mean-value theorem. Since $x^* \in \mathscr C_{x_e}^\circ$, then the reduced Jacobian at $x^*$  is non-singular relative to $\im \Gamma$. Since $x_e-y \in \im \Gamma$, then $y=x_e$: a contradiction.
\end{proof}
\begin{remark} Since the Jacobian is $P_0$, then if arbitrary inflows and outflow were added to every species of a GSN the resulting Jacobian would be a $P$ matrix \cite{berman94}; this is also known as the continuous-flow stirred tank reactor (CFSTR) version of the network. The CFSTR network is injective. This shows that our framework has a direct relationship to the recent results on injectivity \cite{craciun05}, $P$ matrices for reaction networks \cite{banaji07,banaji09} and concordance \cite{feinberg12}.
\end{remark}

\subsection{Persistence}
An important dynamical property in the context of positive systems is \emph{persistence}. Informally it holds if any solution initialized in the interior of the positive orthant will not asymptotically approach its boundary. In \cite{Angeli07} is shown that if a conservative network does not have critical siphons, then it is persistent.  Note that this is a graphical property which is independent of the specific realization of the kinetics involved.


In this section it is shown that persistence can be established for conservative GSNs under suitable conditions detailed in the following Theorem.

\begin{theorem}[Absence of Types of Critical Siphons]\label{th.criticalSiphon} Given \eqref{e.ode}. Consider the network family $\mathscr N_{A,B}$. Assume there exists a critical siphon $P$, and let $\Lambda(P)$ be the set of output reactions of $P$. Then, $\mathscr N_{A,B} $ is a not GS, i.e., the network does not admit a PWLR Lyapunov function if any of the following conditions is satisfied.
\begin{enumerate}
  \item $P$ is a critical deadlock.
   \item  the network is conservative and for some realization of the network family there exists a point  in the interior of a proper stoichiometric compatibility class on which the reduced Jacobian is nonsingular,
  \item  the network is conservative and $\ker \Gamma$ is one-dimensional.
\end{enumerate}
\end{theorem}

The following theorem follows immediately from \cite{Angeli07}[Theorem 2] and Theorem \ref{th.criticalSiphon}.

\begin{theorem}[Persistence of A Class of GSNs]\label{th.special_persistence}
Given \eqref{e.ode}. Assume the network is conservative. Then  the network family $\mathscr N_{A,B} $ is persistent if
 \begin{enumerate}
   \item  $\ker \Gamma$ is one-dimensional,  {or}

    \item by removing the reverse of some reactions the reduced stoichiometry matrix has a one-dimensional kernel and the associated network is GS,   
        {or}
   \item there exists an non-degenerate positive equilibrium in the interior of some proper stoichiometric class for a realization of $\mathscr N_{A,B}$.
 \end{enumerate}

\end{theorem}
The second item in Theorem \ref{th.special_persistence} follows by the fact that the inclusion of a reverse of a reaction does not create a critical siphon.

\subsection{Exponential Stability}
We have shown that the existence of a PWLR Lyapunov function implies that it is a common Lyapunov function for all linear systems that belong to a linear differential inclusion. 

In fact, one of the properties of systems that admits piecewise linear Lyapunov function is that a stable equilibrium can not have purely imaginary eigenvalues \cite{castelan91}. Hence, the reduced Jacobian at a non-degenerate equilibrium can not admit pure imaginary eigenvalues which implies the following Theorem:
\begin{theorem}[Exponential Stability]\label{th.exp_stability} Given \eqref{e.ode} that admits a PWLR function. If a positive equilibrium $x_e$ is non-degenerate relative to $\mathscr C_{x_e}$, then it is exponentially asymptotically stable.
\end{theorem}

\begin{remark} For a conservative network the state space is compact. Therefore, the existence of a non-degenerate equilibrium implies that it is unique (Theorem \ref{cor.p0_app}), it is locally exponantially stable (Theorem \ref{th.exp_stability}), the level sets of the Lyapunov function are always invariant (Theorem \ref{th.lyap_rlf}) and the network is persistent (Theorem \ref{th.special_persistence}). However, global asymptotic stability (GAS) can not be claimed directly without a LaSalle argument. It is potentially possible that the equilibrium is not GAS; for instance, there can be a  limit cycle living in the boundary of the basin of attraction that attracts the outside trajectories. Despite the fact that this seems unlikely, it can not be precluded completely without a proof. Therefore, the graphical algorithm for verifying the LaSalle's argument presented in \cite{PWLRj} is still needed to claim GAS.
\end{remark}

\section{Relationship to Contraction Analysis}
Contraction analysis is an approach to stability investigation focused on the relative behaviour of solutions, rather than on their deviations from a nominal trajectory (such as an equilibrium point). This area of research is as old as the concept of contraction mappping, however, it has sparked growing interest in the control systems community in relationship to the analysis of dynamical systems \cite{lohmiller98}, \cite{sontag14}.

The are several formulations of contraction theory. We are going to present the formulation that utilizes matrix measures, or logarithmic norms.

 \begin{definition}[Logarithmic Norms] For a given induced matrix norm $\|.\|_*$ on $\mathbb R^{n\times n}$, the associated matrix measure (or logarithmic norm) can be defined as follows for a matrix $A \in \mathbb R^{n \times n}$:
\begin{equation}\label{e.lognorm}\mu_*(A):=\limsup_{h \to 0^+} \frac{ \|I+hA\|_*-1}h.\end{equation}
Note that the same definition applies if $\|.\|_*$ is a semi-norm.
\end{definition}
\begin{remark}The logarithmic norm can be evaluated for the standard norms. For instance, the following expression can be used for the $\ell_\infty$ norm:
\begin{equation}\label{e.lognorm_inf}\mu_\infty(A) = \max_{i} \left (a_{ii} + \sum_{j \ne i} |a_{ij}|\right ).\end{equation}
Note that this expression is identical to the one appearing in \eqref{e.lognorm_expression}. \end{remark}

For a dynamical system, negativity of the logarithmic norm can be linked to contraction. This result has been stated in different forms, refer to the tutorial \cite{sontag14} for more details. We state the result as follows.
\begin{theorem}[\cite{sontag14}] \label{th.contraction}Consider a dynamical system $\dot x = f(x)$ defined on a convex subset $X$ of $\mathbb R^n$. Let $| \cdot |_*$ be a norm in $\mathbb R^n$ and $\|.\|_*$ the induced matrix norm on $\mathbb R^{n \times n}$. Assume that
\[ \forall x \in X, \quad \mu_*\left ( \frac{\partial f}{\partial x} (x) \right )  \le c. \]
Then for any two solutions $\varphi(t;x), \varphi(t;y)$ of the dynamical system, the following condition holds:
\begin{equation}
| \varphi(t;x) - \varphi(t;y) |_* \le e^{c t} | \varphi(0;x) - \varphi(0;y) |_*.
\end{equation}
\end{theorem}

Note that if $c<0$  the solutions of the system are exponentially contracting. If $c=0$, then the system is non-expansive. The choice of the norm plays a crucial role even with respect to diagonal weighings. The result above have been applied to CRNs before by weighting the $\ell_1$-norm with a diagonal matrix \cite{DiBernardo10}.  The link between the logarithmic norms and norm-based Lyapunov functions has been established before \cite{kiendl92}. Since convex PWLR Lyapunov functions are $(\infty)$-norms weighted by a  {non-square} matrix, a similar result can be expected to hold. We state in the following Theorem  the precise relationship between convex PWLR functions introduced before and contraction analysis.
\color{black}

\begin{theorem}[Relationship to Contraction Analysis] \strut
\begin{enumerate}
  \item Given the extent of reaction representation of CRNs $\dot \xi = R(x_e + \Gamma \xi)$. Assume that there exists a convex PWLR function $\tV(\xi)=\|C\xi \|_\infty$, and let $\mu_C$ be the logarithmic norm associated. Let the associated Jacobian be: $J_1(\xi):=\frac{\partial R}{\partial x} \Gamma  $. Then,
\[ \forall \xi, \mu_C ( J_1(\xi) ) \le 0. \]
Hence, the system is non-expansive on the subspace $(\ker\Gamma)^{\bot}$, where $\mathbb R^\nu=\ker\Gamma \oplus (\ker\Gamma)^{\bot}$.  
  \item   Given the ODE $\dot x = \Gamma R(x)$. Assume that there exists convex PWL function 
       $V_2(x)=\|B(x-x_e)\|_\infty$, \color{black} and let $\mu_B$ be the logarithmic norm associated. Let the associated Jacobian be: $J_2(x):=\Gamma\frac{\partial R}{\partial x}$. Then,
\[ \forall x \in \mathscr C_{x_e},  \, \mu_B ( J_2(x) ) \le 0. \]
Hence, the system is non-expansive in each stoichiometric class $\mathscr C_{x_e}$.
\end{enumerate}
\end{theorem}

\begin{remark}\label{r.lasalle_contraction} The upper bound $\mu_\infty \left ( \sum_{\ell=1}^s \rho_\ell  \tilde\Lambda^{\ell} \right )$ can be identical to zero especially if $\tfrac m2 \ge n$. Therefore, an analogous concept to a LaSalle argument need to be introduced. This is discussed in the next section.
  \end{remark}

  \subsection{Variational Dynamics and LaSalle Argument for Contraction Analysis}
   In a recent work, Forni and Sepulchre \cite{sepulchre14} have proposed a Lyapunov framework for contraction analysis using a so-called \emph{Finsler structure}. In order to minimize the background needed, we apply it directly to our context. The Finsler-Lyapunov function for the system \eqref{e.extent_ode_e} is $V_F:  \text{\textbf{\textit{T}}} \bar{\mathbf R}_+^\nu \to \bar{\mathbf R}_+$. If contraction analysis was carried out with respect to the norm: $\|.\|_*: \xi \mapsto \|C \xi\|_\infty$, then the corresponding Finsler-Lyapunov function would be
    \[ V_F(\delta \xi) = \| C \delta \xi \|_\infty . \]
    The Finsler structure is given by the mapping $\delta \xi \mapsto \| C \delta \xi \|_\infty$. Therefore, $V_F$ can be considered as a Lyapunov function for the variational system:
    \begin{equation}\label{e.variational_system} \dot{\delta \xi} = \frac{\partial R}{\partial x} \Gamma \delta \, \xi.\end{equation}
The Finsler structure induces a distance function, which is, in this case, $d_F(x,y)=\|C(x-y)\|_\infty$. Hence, if $V_F$ is strictly decreasing then this would imply that this system is incrementally asymptotically stable with respect to the distance function \cite{sepulchre14}, which is equivalent to the result given by Theorem \ref{th.contraction}.
However, if strict decreasingness does not hold, the Finsler-Lyapunov framework for contraction analysis has the advantage of accommodating a LaSalle's invariance principle that can be used to show strict contraction. In fact, the same algorithm proposed in \cite{PWLRj} can be used for the Finsler-Lyapunov function to test the LaSalle's Condition. The following theorem states the result.
\begin{theorem}[Strict Contraction]\label{th.lasalle_contraction} Given \eqref{e.ode}. Assume that $V(x)=\|C R(x)\|_\infty$ is a PWLR Lyapunov function that satisfies Proposition 7 in \cite{PWLRj} . Then the trajectories of \eqref{e.extent_ode_e} are exponentially contractive with respect to the norm $\|.\|_*: \xi \mapsto \|C \xi\|_\infty$ in directions orthogonal to $\ker \Gamma$.
\end{theorem}
\begin{proof} As per \cite[Theorem 2]{sepulchre14}, we need to show that if a trajectory lives in $\ker \dot V_F$, then it is an equilibrium for the variational system \eqref{e.variational_system}. Since $V_F$ is a piecewise function it can be studied per partition regions as before. Hence let $V_F(\delta \xi)=c_k^T \delta \xi$ when $\delta \xi \in \cW_k$. Then,
\[ \dot V_F (\delta \xi ) = c_k^T \dot{\delta \xi} = c_k^T \frac{\partial R}{\partial x} \Gamma{\delta \xi}. \]
Note that the expression above is analogous to (25) in \cite{PWLRj}, where $\sgn(\Gamma \delta \xi)$ can be made constant in each partition region. Therefore, the arguments of the proof of \cite[Proposition 7]{PWLRj} can be replicated to show that the algorithm proposed in \cite{PWLRj} implies that when $\delta \xi (t) \in \ker \dot V_F$ for all $t \ge 0$, then $\Gamma \delta \xi(t) \equiv 0$.
\end{proof}
\begin{remark}Parallel to the duality expounded in the previous sections, the results can also be interpreted for the variational system in species coordinates:
\begin{equation}\label{e.2ndderivative}
\dot{\delta x} = \Gamma  \frac{\partial R(x)}{\partial x} \delta x, D \delta x(0)=0.
\end{equation}
Then, $\hV(\delta x)$ is a Lyapunov function for \eqref{e.2ndderivative}.
\end{remark}
\section{Conclusions}
We have presented a theoretical complement for \cite{PWLRj}. It has been shown that PWLR Lyapunov functions introduced in \cite{PWLRj} can be interpreted as defining a common Lyapunov function for a linear differential inclusion in the reaction coordinates. Furthermore, a dual Lyapunov function in the species coordinates has been defined.
Many additional properties of graphically stable networks have been shown. Examples of biochemical networks for which our framework is applicable can be found in \cite{MA_ecc16}.
\section*{Appendix: Proofs}
Before proving the results of the paper, we need to state and
prove the following Lemma:

 \begin{lemma}\label{lem.lip} Let $\dot x := f(x)$, and let $V: \bar{\mathbb R}_+^n  \to   \bar{\mathbb R}_+$ be a locally Lipschitz function such that:
\[ \frac{\partial  V(x)}{\partial x} f(x) \le 0, \ \mbox{whenever} \ \frac{\partial V(x)}{\partial x}  \ \mbox{exists}, \]
then $\dot{V}(x) \le 0$ for all $x$.
\end{lemma}
\begin{proof}Since $V$ is assumed to be locally Lipschitz, Rademacher's Theorem implies that it is differentiable (i.e., gradient exists) almost everywhere \cite{clarke97}. Recall that for a locally Lipschitz function the \emph{Clarke's gradient} at $x$ can be written as  $\partial_C V(x) :=\co \partial V(x)$, where:
 \begin{align*}\partial V(x):=  \left\{ p \in \mathbb R^n : \exists x_i \to x \, \mbox{with} \, {\partial V(x_i)}/{\partial x} \ \mbox{exists, such that}, p=\lim_{i \to \infty} {\partial V(x_i)}/{\partial x} \right \}. \end{align*}
Let $p \in \partial V(x)$ and let $\{x_i\}_{i=1}^\infty$  be the corresponding sequence. By the assumption stated in the lemma, $({\partial V(x_i)}/{\partial x}) f(x_i) \le 0$, for all $i$. Hence, the definition of $p$ implies that $p^Tf(x) \le 0$. Since $p$ was arbitrary, the inequality holds for all $p \in \partial V(x)$. \\
Now, let $p \in \bar \partial V(x)$ where $ p=\sum_{i} \lambda_i p_i$ is a convex combination of any $p_1,...,p_{n+1} \in \partial V(x)$. By the inequality above, $p^T f(x) = \sum_{i} \lambda_i (p_i^T f(x) ) \le 0. $ Hence, $p^T f(x) \le 0$ for all $p \in \bar \partial V(x)$. \\
As in \cite{clarke97}, the Clarke's derivative of $V$ at $x$ in the direction of $f(x)$ can be written as $D_{f(x)}^C V(x) = \max \{ p^T f(x) : p \in \bar \partial V(x) \} $. By the above inequality, we get $D_{f(x)}^C V(x)  \le 0 $ for all $x$. Since the Dini's derivative is upper bounded by the Clarke's derivative, we finally get:
\begin{align*}\dot V(x)&:= \limsup_{h \to 0^+} \frac{ V(x+h f(x)) - V(x)} h  \le \underset{y \to x}{\limsup_{h \to 0^+}} \frac{ V(y+h f(x)) - V(y)} h =: D_{f(x)}^C V(x) \le 0, \end{align*}
for all $x$.
\end{proof}

\paragraph*{Proof of Theorem \ref{th.uncertain}} We show the existence of the common Lyapunov function implies the existence of the RLF. Nonnegativity of $V$ follows from the nonnegativity of $\tV$. Let $(i,j)=\kappa(\ell)$, recall that $\Gamma^\ell=e_j \gamma_i^T$, hence $\ker \tV=\bigcap_{\ell=1}^s \ker \Gamma^{\ell} = \ker \Gamma$. Therefore, $ R(x) \in \ker V$ iff $\Gamma R(x)=0$, which establishes the positive-definiteness of $V$.

We assumed that $\tV$ has a negative semi-definite time-derivative for every linear system in the considered set. Hence, when $\tV$ is differentiable, we can write  $({\partial \tV}/{\partial r}) \Gamma^\ell r \le 0$, $\ell=1,...,s$. Hence, for any $\rho^1,...,\rho^s \in \bar{\mathbb R}_+$:
\begin{equation}\label{e.commonlyap_dec} \sum_{\ell=1}^s \rho^\ell \frac{\partial \tV}{\partial r} \Gamma^\ell r \le 0, \mbox{when} \, (\partial V(r) / \partial r) \ \mbox{exists}. \end{equation}
Therefore, when $\tV$ is differentiable:
\begin{align}
  \dot V(x) &  = \frac{\partial \tV}{\partial R}
   \frac{\partial R}{\partial x}(x) \Gamma R( x  )  \label{e.Vodt_expansion}
       = \frac{\partial \tV}{\partial R}
   \left ( \sum_{i,j: \alpha_{ij}>0}  \frac{\partial R_j}{\partial x_i}(x) E_{ji} \right )\Gamma R( x  ),
\end{align}
where $\partial \tV/\partial R:= \left . (\partial \tV/\partial r) \right |_{r=R(x)}$. \\ Now, denote $\rho^\ell= \frac{\partial R_j}{\partial x_i}(x) $, which is nonnegative by A3. This allows us to write:
\begin{align} \label{e.uncertain_proof0}
  \dot V(x) &  =  \sum_{\ell=1}^s \rho^\ell \frac{\partial \tV}{\partial R} E_{ji} \Gamma R(x)\\ \label{e.uncertain_proof}
   & =\sum_{\ell=1}^s \rho^\ell \frac{\partial \tV}{\partial R} \Gamma^\ell R(x) \le 0,\, \mbox{for almost all} \, x.
\end{align}
The last inequality follows from \eqref{e.commonlyap_dec}. Using Lemma \ref{lem.lip}, $\dot V(x)\le 0$ for all $x$, and for all $R \in \mathscr K_A$.

 In order to show the other direction, almost all properties outlined in Definition \ref{def.commonLyap} are clearly satisfied, we just show nonincreasingness. Assume that there exists $\ell$ such that $\tV(r)$ is not nonincreasing along the trajectories of $\dot r=\Gamma^\ell r$. Consider the corresponding term in \eqref{e.uncertain_proof0}. Since $V(R(x))$ is a Lyapunov function for any choice of admissible rate reaction function $R$, choose $\rho^\ell=\frac{\partial R_j}{\partial x_i}$ to be large enough such that $\dot V(x) \ge 0$ for some $x$; a contradiction.
 \hfill $\blacksquare$

\paragraph*{Proof of Theorem \ref{th.pwlr_dual}} The first statement follows from Theorem \ref{th.dual_uncertain}. In order to show the second statement, let $V_2(x)= b_k^T (x-x_e),$ for $x-x_e \in \mathcal V_k$. We will show that $V_1(x)= c_k^T R(x),$ for $R(x) \in \cW_k$ is nondecreasing. Without loss of generality, the partition matrix can be written in the form: $G=[ I \  \hat G^T]^T$. This representation implies that the sign of $x-x_e$ is determined in every region $\mathcal V_k$. Now, assume that $x-x_e \in \mathcal V_k^\circ$, then:
\[ \!\!\dot V_2(x)=\! b_k^T \Gamma R(x) \!= \! c_k^T R(x) \! \le \! 0 = c_k^T R(x_e),\, \mbox{for all} \, R\in \mathscr K_A. \]
Let $R_j(x) \in \supp R(x)$, and let $\alpha_{ij}>0$. Since $R$ is nondecreasing by A3, if $\sgn(x_i-x_{e_i})\sgn(c_{kj})>0$, there exists $R \in \mathscr K_A$ such that $\dot V_2(x) \ge 0$. Hence, this implies that the inequality $\sgn(c_{kj})\sgn(x_i-x_{e_i})\le 0$ holds. Fix $j$, if there exists $i_1,i_2$ such that $\alpha_{i_1 j}, \alpha_{i_2j}>0$ and $\sgn(x_{i_1}-x_{e_{i_1}})\sgn(x_{i_2}-x_{e_{i_2}})<0$, then $\sigma_{kj}:=0$. Otherwise, $\sigma_{kj}:=\sgn(x_i - x_{e_i})$ for some $i$ such that $\alpha_{ij}>0$. \\
Hence, in order to have $\dot V_2(x) \le 0$ for all $R \in \mathscr K_A$ we need that $ \sigma_{kj} (x_i - x_{e_i}) \ge 0 $ whenever $x-x_e \in \mathcal V_k$, for all $k,j,i$ with $\alpha_{ij}>0$. By Farkas Lemma \cite{rockafellar}, this is equivelant to the existence of $\lambda_{kji} \in \bar{\mathbb R}_+^{p}, \zeta_{kji} \in \mathbb R^\iota$:
\begin{equation}\label{e.dual_farkas} \sigma_{kj} e_i^T = \lambda_{kji} ^ T \Sigma_k G + \zeta_{kji}^T D,    \end{equation}
where $D^T \in \mathbb R^{\iota \times n}$ is a matrix whose columns are basis vectors for $\ker \Gamma^T$. \\
If we multiply both sides of \eqref{e.dual_farkas} by $\Gamma$ from the left, then we get condition C4 in \cite[Theorem 4]{PWLRj} which necessary and sufficient for $\dot V_1(x)= \frac{d}{dt}(c_k^T R(x)) \le 0$. \hfill $\blacksquare$

 \paragraph*{Proof of Theorem \ref{th.criticalSiphon}}
 Assume $P$ is a critical siphon for the petri-net associated with $\Gamma$, and let $n_p=|P|$. Let $\Lambda(P)$ be the set of output reactions of $P$, and let $\nu_p=|\Lambda(P)|$.

Before we prove item 1 of Theorem \ref{th.criticalSiphon}, the following lemma is needed.

\begin{lemma}\label{lem.deadlock_region} Consider a network family $\mathscr N_{A,B}$. Let $P$ be a set of species that does not contain the support of a conservation law; let its indices be numbered as $\{1,...,n_p\}$. Then, there exists a nonempty-interior region $\{r | \Sigma_k \Gamma r \ge 0\} $ with a signature matrix $\Sigma_k$ that satisfies $\sigma_{k1}=...=\sigma_{kn_p}=1$.
\end{lemma}
\begin{proof} Assume the contrary. This implies that $ \cap_{i=1}^{n_p} \{ R | \gamma_{i } ^ T R > 0 \} \, \bigcap \, \cap_{i=n_p+1}^{n} \{ R | \sigma_{i} \gamma_{i } ^ T R > 0 \} = \emptyset$ for all possible choices of signs $\sigma_i=\pm 1$. However, $\mathbb R^r$ can be partitioned into a union of all possible half-spaces of the form $\cap_{i=n_p+1}^{n} \{ R | \sigma_{i} \gamma_{i } ^ T R \ge 0 \}$. Therefore, this implies that $ \cap_{i=1}^{n_p} \{ R | \gamma_{i } ^ T R > 0 \} = \emptyset$. By Farkas Lemma, this implies that there exists $\lambda \in \mathbb R^t$ satisfying $\lambda > 0$ such that $ [\lambda^T 0] \Gamma =0$. Therefore, $P$ contains the support of the conservation law $[\lambda^T \,  0]^T$; a contradiction.\end{proof}

Therefore, we can state the proof of the first item:
\begin{proof}[{Proof of Theorem \ref{th.criticalSiphon}-1)}]
Without loss of generality, let $\{1,...,n_p\}$ be the indices of the species in $P$. Using Lemma \ref{lem.deadlock_region}, there exists a nonempty-interior sign region $\mathcal S_k, 1 \le k \le m_s$ with a signature matrix $\Sigma_k$ that satisfies $\sigma_{k1}=...=\sigma_{kn_p}=1$.
  Since $\Lambda(P)=\mathscr R$, this implies $b_{kj}\le 0$ for all $j=1,..,\nu$. However, this is not allowable by \cite[Theorem 9]{PWLRj} since $ \zeta_k^T B_k v \le 0 $ for all $v \in \ker \Gamma \cap \Rnn$ and for any choice of admissible $\zeta_k$.
\end{proof}

In order to proceed, an existence result of equilibria is needed:
\begin{lemma}\label{lem.brouwer}
 Consider a network family $\mathscr N_{A,B}$. Let $P$ be a critical siphon, and let $\Psi_P$ be the associated face. If the network is conservative, then for any proper stoichiometric compatibility $\mathscr C$, there exists an equilibrium $x_e$ of \eqref{e.ode} such that $x_e \in \Psi_P \cap \mathscr C$.
\end{lemma}
\begin{proof} The set $\Psi_P \cap \mathscr C$ is compact, forward invariant, and convex, since the both sets $\Psi_P, \mathscr C$ are as such. Hence, the statement of the lemma follows directly from the application of the Brouwer's fixed point theorem on the associated flow.
\end{proof}

We are ready now to prove the second item of Theorem \ref{th.criticalSiphon}.
\begin{proof}[ {Proof of Theorem \ref{th.criticalSiphon}-2)}]By Lemma \ref{lem.brouwer}, there exists an equilibrium in $\Psi_P$. Since it is assumed that there exists an isolated equilibrium in interior, Theorem \ref{cor.p0_app} implies that $\mathscr N_{A,B}$ is not GS.
\end{proof}

Before concluding the proof, a simple lemma is stated and proved:
\begin{lemma}\label{lem.zero_siphon}
Let $x_e$ be an equilibrium of \eqref{e.ode}. Let $\tilde P$ be a set of species that correspond to $\{1,..,n\} \backslash \supp(x_e)$. Then, $\tilde P$ is a siphon.
\end{lemma}
\begin{proof} Assume that $\tilde P$ is not a siphon, then there exists some $X_i \in \tilde P$ and $\mathbf R_j \in \mathscr R$ such that $X_i$ is a product of $\mathbf R_j$ and $\mathbf R_j \ne \Lambda (\tilde P)$. At the given equilibrium, all negative terms in the expression of $\dot x_i$ vanish since $x_{ei}=0$. Since $X_i$ is not a reactant in $\mathbf R_j$ this implies $\beta_{ij}>0, \alpha_{ij}=0$ then $R_j(x)$ has a strictly positive coefficient which implies $\dot x_i >0$; a contradiction.
\end{proof}

Hence, we are ready to conclude the proof of Theorem \ref{th.criticalSiphon}:
\begin{proof}[{Proof of Theorem \ref{th.criticalSiphon}-3)}]
 By Lemma \ref{lem.brouwer}, there exists an equilibrium $x^* \in \Psi_P$ such that $\Gamma R(x^*)=0$. Since $\dim(\ker \Gamma)=1$, this implies that $R(x^*)=t v$ for some $t \ge 0$. Consider the case $t=0$. This implies $R(x^*)=0$. Then, $P \subset \tilde P:= \{1,..,n\}\backslash \supp(x^*)$. $\tilde P$ is a siphon by Lemma \ref{lem.zero_siphon}, and since  $P \subset \tilde P$ it is a critical deadlock. However, by Theorem \ref{th.criticalSiphon}-1), $\mathscr N_{A,B}$ is not GS. If $t>0$, this implies that $P=\emptyset$; a contradiction.
\end{proof}

\paragraph*{Proof of Theorem \ref{th.contraction}}
Write the logarithmic norm expression using \eqref{e.lognorm}:
\begin{align*}
     \mu_C(J_1(\xi)) & = \limsup_{h \to 0^+} \frac 1h \left ( \left \|    I + h \frac{\partial R}{\partial x} \Gamma   \right \|_C - 1  \right )
\end{align*}
The expression above includes the induced matrix norm. 
Using the definition of the induced matrix norm we proceed as follows:
\begin{align*}
       \left \|   I + h \frac{\partial R}{\partial x} \Gamma \xi   \right \|_C & = \sup_{\|C\xi\|_\infty =1 } \left \| C \left ( I + h \frac{\partial R}{\partial x} \Gamma  \right ) \xi \right \|_\infty
        \mathop{=}^{(\star)}   \sup_{\|C\xi\|_\infty =1 } \left \| C \xi + h   \sum_{\ell=1}^s \rho_\ell(x) C e_j \gamma_i^T \xi   \right \|_\infty
       \\ & \mathop{=}^{(\clubsuit)}    \sup_{\|C\xi\|_\infty =1 } \left \| C \xi + h   \sum_{\ell=1}^s \rho_\ell  \tilde\Lambda^{\ell} C \xi   \right \|_\infty
       \\ &  \le      \sup_{\|C\xi\|_\infty =1 } \left \| I + h   \sum_{\ell=1}^s \rho_\ell  \tilde\Lambda^{\ell}     \right \|_\infty \|C \xi \|_\infty  \left \| I + h   \sum_{\ell=1}^s \rho_\ell  \tilde\Lambda^{\ell}     \right \|_\infty,
\end{align*}
where $(\star)$ is by \eqref{e.rho_decomp} and $(\clubsuit)$ is by \eqref{e.NewCheck_tilde}. Therefore, the expression of the logarithmic norm above can be written as:
\begin{equation}\label{e.lognorm_ineq}  \mu_C(J_1(\xi))  \le \mu_\infty \left ( \sum_{\ell=1}^s \rho_\ell  \tilde\Lambda^{\ell} \right ) \le   \sum_{\ell=1}^s \rho_\ell  \mu_\infty(\tilde\Lambda^{\ell}) = 0, \end{equation}
where the inequalities follow by the subadditivity of the logarithmic norm and \eqref{e.lognorm_expression}.

Note since $C$ has nonempty kernel space, then $\|C\xi\|_\infty$ is a semi-norm. However, Theorem \ref{th.contraction} requires a norm. This can be remedied by studying the system in directions orthogonal to $\ker \Gamma$ by defining a transformation of coordinates using a matrix $T_1= [ \hat T_1,v_1,..,v_{\nu-\nu_r } ]^T $, where $\{v_1,..,v_{\nu-\nu_r }\}$ is a basis of $\ker \Gamma$, $\nu_r=rank(\Gamma)$, and $\hat T_1$ is chosen so that $T_1$ is invertible. Defining $\hat \xi=T\xi$, then the first $\nu_r$ coordinates of $\hat\xi$ are decoupled from the rest. Hence, inequality \eqref{e.lognorm_ineq} can be established similarly for the reduced Jacobian which is the upper right $\nu_r \times \nu_r$ block of $T_1 J_1(\xi) T_1^{-1}$. The norm in the reduced subspace is $\|.\|_{\hat C}: z \mapsto \|\hat C z \|_\infty, z \in \mathbb R^\nu$, where $\hat C$ is $\tfrac m2 \times \nu_r$ defined as the nonzero columns of $C T_1^{-1}$. The equations $ C e_j \gamma_i^T = \tilde\Lambda^{\ell} C$ are equivalent to $(C T_1^{-1} )(T_1 e_j \gamma_i^T T_1^{-1}) = \tilde\Lambda^{\ell} (CT_1^{-1})$. Therefore, everything goes through in the reduced space, and the same upper bound is valid.

The same argument can be replicated for to prove the second item in the theorem. This is accomplished by utilizing the alternative representation \eqref{e.alternative_representation}, the rank-one decomposition of the Jacobian \eqref{e.rho_dual}, and noting that conditions \eqref{e.NewCheck_tilde} can be written as: $B\Gamma_i e_j^T = \tilde\Lambda^\ell B + Y^{\ell} D^T$ for some matrices $Y^{1},..,Y^{\ell}$. The reduced space argument can be carried out also by using a transformation matrix $T_2=[\hat T_2, D]^T$. \hfill $\blacksquare$

\section*{References}

\end{document}